\documentclass[12pt,leqno]{article}

\usepackage{amsmath,amsthm,amsfonts,amssymb,times,hyperref,afterpage,enumitem}
\usepackage{xcolor}
\definecolor{Burgundy}{RGB}{128,0,32}

\allowdisplaybreaks

\newtheorem{theorem}{Theorem}
\newtheorem{lemma}{Lemma}

\allowdisplaybreaks

\begin{document}

\centerline{\Large \textbf{A generalization of a fourth}}
\vskip 4pt
\centerline{\Large \textbf{irreducibility theorem of I.~Schur}}

\vskip 10pt
\centerline{by}
\vskip 10pt
\centerline{\large \textsc{Martha Allen} (Milledgeville, GA) and}
\vskip 4pt
\centerline{\large \textsc{Michael Filaseta} (Columbia, SC)}

\renewcommand{\thefootnote}{}
\vskip 0pt

\section{Introduction}

In 1929, I.~Schur \cite{sc} showed the following general theorem.

\begin{theorem}[Schur]
Let $n$ be a positive integer, and let $a_0, a_1, \dots, a_n$ denote 
arbitrary integers with $|a_n|=|a_0|=1$.  Then
\[
f(x) = a_{n}\dfrac{x^{n}}{n!} + a_{n-1}\dfrac{x^{n-1}}{(n-1)!} + \cdots + a_{2} \dfrac{x^{2}}{2!} + a_{1} x + a_{0} 
\]
is irreducible.   
\end{theorem}

\noindent
Here, and throughout this paper, irreducibility is over the field of rational numbers.  

The second author \cite{fi} showed that the condition that $|a_{n}| = 1$ can be relaxed to $0 < |a_{n}| < n$ provided
$(a_{n}, n) \not\in \{ (\pm 5, 6), (\pm 7, 10) \}$, and that this result is in some sense best possible. 
More precisely, there are examples for each pair $(a_{n}, n) \in \{ (\pm 5, 6), (\pm 7, 10) \}$ where
$0 < |a_{n}| < n$, $|a_0|=1$ and the polynomial $f(x)$ is reducible, and there are examples of reducible $f(x)$ for each
$n > 1$ with $a_{n} = \pm n$ and $|a_0|=1$.  

In a second paper by I.~Schur \cite{sch}, three other similar irreducibility results were obtained involving 
a condition $|a_{n}| = 1$ on part of the leading coefficient as before.  Two of these results were generalized by the
authors in \cite{af,af2} as follows.

\begin{theorem}
For $n$ an integer $\ge 1$, define 
\[
f(x)=\sum_{j=0}^{n}a_j\frac{x^j}{(j+1)!}
\]
where 
the $a_j$'s are arbitrary integers with $|a_0|=1$.
Let $k'$ be the integer such that
$n+1=k'2^u$ where $k'$ is odd and $u$ is an integer $\ge 0$. 
Let $k''$ be the integer such that
$(n+1)n=k''2^u3^v$ where $(k'',6)=1$, $u$ is an integer $\ge 1$, and $v$ 
is an integer $\ge 0$.  Let $M=\min \{ k',k'' \}$.  If
$0 < |a_n| < M$, then $f(x)$ is irreducible.
\end{theorem}

\begin{theorem}
For $j \ge 0$, let $u_{2j}=1\times 3\times 5\times \cdots \times(2j-1)$. 
For $n$ an integer $>1$, define 
\[
f(x)=\sum_{j=
0}^{n}a_j\frac{x^{2j}}{u_{2j}}
\]
 where the $a_j$'s are arbitrary integers with $|a_0|=1
$.  If $0 < |a_n| < 2n-1$, then $f(x)$ is irreducible for all but finitely
many pairs $(a_n,n)$.
\end{theorem}

In this paper, we generalize the fourth irreducibility theorem of Schur~\cite{sch} from 1929 
in which again a condition $|a_{n}| = 1$ on part of the leading coefficient is relaxed.  
Specifically, we prove the following.

\begin{theorem}\label{themainthm}
Let $u_{2j}=1 \times 3 \times  5\ \times \cdots \times (2j-1)$. 
For $n$ an integer $\ge 1$, define 
\begin{equation}\label{themainthmeq}
f(x)=\sum_{j=0}^{n}a_j\frac{x^{2j}}{u_{2j+2}}
\end{equation}
where the $a_j$'s are arbitrary integers with $|a_0|=1$.
Let $k'$ be the integer such that 
$2n+1=k'3^u$ where $u$ is an integer $\ge 0$ and $(k',3)=1$.
Let $k''$ be the integer such that 
$(2n+1)(2n-1)=k''3^u5^v$ where $u$ and $v$ are integers $ \ge 0$ 
and $(k'',15)=1$.  
Let $M= \min \{ k', k'' \}$.
If $0 < |a_n| < M$, then
$f(x)$ is irreducible for 
all but finitely many pairs $(a_n,n)$.
\end{theorem}

In the case of $|a_n|=1$, I.~Schur \cite{sch} showed that $f(x)$ above is irreducible unless $2n = 3^{u}-1$ for some integer $u \ge 2$.
Furthermore, he showed that if $f(x)$ is reducible with $|a_n|=1$, then $f(x)$ is $x^{2} \pm 3$ times an irreducible polynomial.  
Observe that Theorem~\ref{themainthm} gives no information in the case that $2n = 3^{u}-1$ as $M = k' = 1$.  
In the case that $u = 1$ so that $f(x)$ has degree $2$, the quadratic will be irreducible.  
We will show that if $f(x)$ is as in \eqref{themainthmeq} with $a_{0} = 1$ and $a_{n} = M = k'$ in general, 
then $f(x)$ is either irreducible or divisible by a quadratic polynomial, and the latter can happen.  
The value of $M$ is also $1$ when $k'' = 1$ which occurs only for $n \in \{ 1, 2, 13 \}$.  
The case $n = 1$ corresponds to $f(x)$ being a quadratic as discussed above. 
For $n = 13$, we already know $f(x)$ can have a quadratic factor since $2 \cdot 13 = 3^{3} - 1$.  
For $n \ge 2$ in general, we show that 
if $f(x)$ is as in \eqref{themainthmeq} with $a_{0} = 1$ and $a_{n} = M = k'' < k'$, 
then $f(x)$ is either irreducible or divisible by a quartic polynomial, and the latter can happen.   
In the case $n = 2$, the polynomial $f(x)$ is quartic and will be irreducible.

From henceforth, the polynomial $f(x)$ will be as defined in \eqref{themainthmeq}.  
Theorem~\ref{themainthm} goes back to work associated with the first author's dissertation~\cite{ma}.  
The result has the weakness over the prior results stated above in that there are
finitely many, possibly zero, exceptional pairs $(a_n,n)$ that are not determined.  
As determining these finitely many pairs seems particularly difficult due to an application of
an ineffective result of Mahler~\cite{mahler}, we are left with this
less than precise result.  On the other hand, like with the prior results, as noted above, we will show that
Theorem~\ref{themainthm} is best possible in the sense that for each $n > 2$, if $a_n=M$ and 
$a_{0}=1$, then there are integers 
$a_{n-1}, a_{n-2}, \dots, a_1$ such that $f(x)$ is reducible.  

Before closing this introduction, we note that recent work related to the fourth irreducibility theorem of Schur~\cite{sch} from 1929
has been done by A.~Jakhar \cite{jakhar}.  

%%%%%%%%%%%%%%%%%%%%%%
\section{The basic strategy}
To establish that $f(x)$ is irreducible for all but finitely many pairs
$(a_n,n)$ with $0 < |a_n| < M$, it suffices to show that $f(x)$ is irreducible for sufficiently
large $n$. We use the following two lemmas in proving the irreducibility of 
$f(x)$ when $0 < |a_n| < M$.  We explain the proof of Theorem~\ref{themainthm} 
based on these lemmas after we state them.  The proofs of the lemmas are in the next two sections.

\begin{lemma}\label{3l1}
Let $a_0, a_1, \dots, a_n$ denote arbitrary integers with $|a_0|=1$, and let 
\[ 
f(x)=\sum_{j=0}^{n}a_j \frac{x^{2j}}{u_{2j+2}}. 
\] 
 Let $k$ be a positive odd integer $ \le n$.  
Suppose there exists a prime $p > k+2$ (so $p \ge k+4$) 
and a positive integer $r$
for which 
\[ 
p^r | \big( (2n+1)(2n-1)(2n-3) \cdots (2n-k+2) \big)
\text{ \ \ and \ \ } p^r \nmid a_n.
\]  
Then $f(x)$ cannot have a factor of degree $k$ or $k+1$.  
\end{lemma}

\noindent
\textit{Comment.} \ Lemma \ref{3l1} implies that if $f(x)$ has a 
factor of degree
$k$ or $k+1$, then
\[
\prod_{\substack{p^r \Vert ( (2n+1)(2n-1) \cdots (2n-k+2) ) \\  p \ge
 k+4 } } p^r \quad \text{ divides } \quad a_n,
 \]
 where the notation $p^{r} \Vert m$ denotes that $p^{r} \mid m$ and $p^{r+1} \nmid m$.

\begin{lemma}\label{3l2}
Let $n$ be a sufficiently large integer, and let $k$ be an odd integer in 
$[5,n]$. Then 
\[
\prod_{\substack{p^r \Vert ( (2n+1)(2n-1) \cdots (2n-k+2) ) \\  p \ge
 k+4 } } p^r>2n+ 1.
 \]
\end{lemma}

Observe that once the above lemmas are established, we can deduce the following consequences of these lemmas.
\begin{itemize}[leftmargin=15pt]
\item If $n$ is a sufficiently large integer and $0 < |a_n| \le 2n+1$, then
$f(x)$ cannot have a factor with
degree $\ell $ in $[5,n]$.
\item If $0 < |a_n| < k''$ where $k''$ is the integer such that 
$(2n+1)(2n-1)=k'' 3^u5^v$ with
$u \ge 0$, $v \ge 0$ and $(k'', 15)=1$, then $f(x)$ cannot have a 
cubic or quartic factor.
\item If $0 < |a_n| < k'$ where $k'$ is the integer such that 
$2n+1=k' 3^u$ with
$u \ge 0$ and $(k', 3)=1$, then $f(x)$ cannot have a linear or quadratic factor.
\end{itemize}

\noindent 
Taking $M = \min\{k',k''\}$ as in Theorem~\ref{themainthm} and noting $k' \le 2n+1$, we obtain that if
$n$ is sufficiently large and 
$0 < |a_n| < M$, then $f(x)$ is irreducible, 
completing the proof of Theorem~\ref{themainthm}.

%%%%%%%%%%%%%%%%%%%%%%
\section{Proof of Lemma \ref{3l1}}
In this section, we use Newton polygons to show that if there is a prime
$p \ge k+4$ and a positive integer $r$ such that $p^r \mid \big( (2n+1)(2n-1) \cdots
(2n-k+2) \big)$ but $p^r \nmid a_n$, then $f(x)$ cannot have a factor of 
degree $k$ or $k+1$.  Here $k$ is an odd positive integer $\le n$.
Before proceeding with the proof, we give some notation 
and background information on Newton polygons that will be useful.

If $p$ is a prime and $m$ is a nonzero integer, we define $\nu(m)= \nu_p(m)$ to be the 
nonnegative integer such that $p^{\nu(m)} \mid m$ and $p^{\nu(m)+1} \nmid m$.
Let $w(x)=\sum_{j=0}^{n} u_j x^{j} \in \mathbb{Z} [x]$ with $u_n u_0 \ne 0$, 
and let $p$ be a prime.  Set
\[
S= \{ \big( n-i,\nu(u_{i}) \big):  u_{i} \ne 0, 0 \le j \le n \}.
\]
Consider the lower edges along 
the convex hull of these points.  The left-most endpoint is $\big(0,\nu(u_n)\big)$ 
and the right-most endpoint is $\big(n,\nu(u_0)\big)$.  The endpoints of each edge 
belong to $S$, and the slopes of the edges increase from left to right.
When referring to the ``edges'' of a Newton polygon, we shall not allow
two different edges to have the same slope.  
The polygonal path formed by these edges is called the Newton polygon of 
$w(x)$ with respect to the prime $p$.  We will refer to the points in 
$S$ as spots of the Newton polygon. 

In investigating irreducibility with Newton polygons, we will make use of 
the following lemma.

\begin{lemma}[G.~Dumas \cite{d}]\label{NPlem}
Let $g(x)$ and $h(x)$ be in $\mathbb{Z} [x]$ with $g(0)h(0)\ne0$, and let 
$p$ be a prime.  Let $k$ be a non-negative integer such that $p^k$ divides 
the leading coefficient of $g(x)h(x)$ but $p^{k+1}$ does not.  Then the edges 
of the Newton polygon for $g(x)h(x)$ with respect to $p$ can be formed by 
constructing a polygonal path beginning at $(0,k)$ and using translates of 
the edges in the Newton polygon for $g(x)$ and $h(x)$ with respect to the 
prime $p$ (using exactly one translate for each edge).  Necessarily, the 
translated edges are translated in such a way as to form a polygonal path 
with the slopes of the edges increasing.      
\end{lemma}

\begin{proof}[Proof of Lemma \ref{3l1}.]

Let
\[
F(x)=u_{2n+2}f(x)= \sum_{j=0}^{n}a_j \frac{u_{2n+2}}{u_{2j+2}}x^{2j}=
\sum_{j=0}^{n}b_jx^{2j},
\] 
where 
\[
b_j=a_j \frac{u_{2n+2}}{u_{2j+2}} = a_j (2n+1)(2n-1) \cdots (2j+3).
\]
Note that the coefficients of the odd powers of $x$ in $F(x)$ are all zero.
Writing these terms into $F(x)$, we set
\[
F(x)=\sum_{j=0}^{n}b_jx^{2j}= \sum_{i=0}^{2n} c_i x^i.
\]
If $i$ is odd, then $c_i = 0$.  On the other hand, if $i$ is even, then
$i=2j$ for some $j \in \{ 0, 1, \dots, n \}$ and $c_i = c_{2j} = b_j$.
We consider the Newton polygon of $F(x)$ with respect to the prime $p$.
The spots of the Newton polygon are 
\[ 
\big\{ \big(2n-i, \nu(c_i)\big): c_{i} \ne 0, 0 \le i \le 2n \big\}.
\]  
From the  conditions that $k$ is odd and 
$p^r \mid \big( (2n+1)(2n-1) \cdots  (2n-k+2) \big)$,
we deduce that $p^r \mid c_i$ for $i \in \{ 0,1,\ldots, 2n-(k+1), 2n-k \}$.
Thus, the right-most spots in 
\[
\mathcal R = \big\{ \big(2n - i, \nu(c_{i})\big) : c_{i} \ne 0, 0 \le i \le 2n-k \big\},
\] 
associated with the Newton polygon of $F(x)$ with respect to $p$, have $y$-coordinates $\ge r$. 
Since $p^r \nmid a_n$ and $c_{2n} = b_n = a_n$, we have $p^r \nmid c_{2n}$.  
Thus, the left-most endpoint of the Newton polygon of $F(x)$ with respect to $p$, 
which is $\big(0, \nu(c_{2n})\big)=\big(0, \nu(a_{n})\big)$, has $y$-coordinate $<r$.
Since the slopes of the edges of a Newton polygon increase
from left to right, the spots in $\mathcal R$ all lie on or above edges of the 
Newton polygon of $F(x)$ with respect to $p$ which have a positive slope.  
We will show next that each of these positive slopes is $< 1/(k+1)$ by showing that the 
right-most edge has slope $< 1/(k+1)$.  

Observe that the slope of the right-most edge of the Newton polygon of $F(x)$ with respect to $p$ is given by
\[
\max_{1\le j \le n}\bigg\{ \frac{\nu(c_0)-\nu(c_{2j})}{2j}\bigg\}
= \max_{1\le j \le n}\bigg\{ \frac{\nu(a_0u_{2n+2})-\nu(a_ju_{2n+2}/u_{2j+2})}{2j}\bigg\}.
\]
Since $|a_{0}| = 1$, we have $\nu(a_0u_{2n+2})= \nu(u_{2n+2})$.  
Also, for $1 \le j \le n$, we see that
\[
\nu\Big(a_j \frac{u_{2n+2}}{u_{2j+2}}\Big) \ge \nu\Big(\frac{u_{2n+2}}{u_{2j+2}}\Big).
\]
Thus,
\begin{align*}
\nu(a_0u_{2n+2})-\nu \Big(a_j \frac{u_{2n+2}}{u_{2j+2}} \Big) 
&\le \nu(u_{2n+2}) -\nu \Big( \frac{u_{2n+2}}{u_{2j+2}} \Big ) \\[5pt]
&= \nu(u_{2j+2}) \le \nu((2j+1)!).
\end{align*}
Thus, the right-most slope is
\begin{equation}\label{lem1eq1}
\max_{1\le j \le n}\bigg\{ \frac{\nu(c_0)-\nu(c_{2j})}{2j}\bigg\}
\le \max_{1\le j \le n}\bigg\{ \frac{\nu((2j+1)!)}{2j} \bigg\}.
\end{equation}

To estimate the right-hand side of \eqref{lem1eq1} further, we consider the two cases
$j < (p-1)/2$ and $j \ge (p-1)/2$.  Suppose first $j < (p-1)/2$.  Then $2j+1 < p$.  Since $p$
is a prime, we see that $p \nmid (2j+1)!$.  Therefore,
$\nu((2j+1)!)=0$.  Therefore, 
\begin{equation}\label{lem1eq2}
\frac{\nu((2j+1)!)}{2j}=0
\qquad \text{for $j < (p-1)/2$}.
\end{equation}
Now, suppose $j \ge (p-1)/2$.  
In this case, we have
\[
\nu((2j+1)!) = \sum_{i=1}^{\infty} \bigg\lfloor \dfrac{2j+1}{p^{i}} \bigg\rfloor
< \sum_{i=1}^{\infty} \dfrac{2j+1}{p^{i}} 
= \dfrac{2j+1}{p-1}.
\]
Since $j\ge (p-1)/2$ implies $1/(2j) \le 1/(p-1)$, we obtain
\begin{equation}\label{lem1eq3}
\begin{aligned}
\frac{\nu((2j+1)!)}{2j} 
&<  \frac{2j+1}{2j} \cdot \frac{1}{p-1} 
= \bigg( 1 + \dfrac{1}{2j} \bigg) \frac{1}{p-1} \\[5pt]
&\le \bigg( 1 + \dfrac{1}{p-1} \bigg) \frac{1}{p-1}
= \frac{p}{(p-1)^2}
\qquad \text{for $j \ge (p-1)/2$}.
\end{aligned}
\end{equation}
Combining \eqref{lem1eq1}, \eqref{lem1eq2} and \eqref{lem1eq3}, we deduce the right-most slope is
\[
\max_{1\le j \le n}\bigg\{ \frac{\nu(c_0)-\nu(c_{2j})}{2j}\bigg\}
< \dfrac{p}{(p-1)^2}.
\]  
Recalling that we have the condition $p \ge k+4$ in the statement of Lemma~\ref{3l1}, one can verify that
$p/(p-1)^{2} < 1/(k+1)$.  
Therefore, the slope of the right-most edge is $< 1/(k+1)$, and we deduce that 
each edge of the Newton polygon of $F(x)$ with respect to $p$ has slope $< 1/(k+1)$.

Now, assume $F(x)$ has a factor $g(x) \in \mathbb{Z}[x]$ with $\deg g \in \{ k, k+1 \}$.  
We will show that the translates of all 
the edges of the Newton polygon of $g(x)$ with respect to $p$ cannot be
found among the edges of the Newton polygon of $F(x)$ with respect to $p$.
This will imply a contradiction to Lemma~\ref{NPlem}.  Hence, it will follow
that $F(x)$ cannot have a factor of degree $k$.

First, we show that no translate of an edge for $g(x)$ can be found among 
those edges in the Newton polygon of $F(x)$ having positive slope.  Suppose
$(a,b)$ and $(c,d)$ with $a<c$ are two lattice points on an edge of the 
Newton polygon of $F(x)$ having positive slope.  We know the slope is 
$< 1/(k+1)$; therefore, 
\[ 
\frac{1}{c-a} \le \frac{d-b}{c-a} < \frac{1}{k+1}.
\]
Thus, $c-a > k+1 \ge \deg g$ so that $(a,b)$ and $(c,d)$ cannot
be the endpoints of a translated edge of the Newton polygon of $g(x)$;
therefore, the translates of the edges of the Newton polygon of $g(x)$ 
with respect to $p$ must be among the edges of the Newton polygon of $F(x)$
having 0 or negative slope.  

Next, we show that not all the  translates of the edges for $g(x)$ can
be found among the edges of the Newton polygon of $F(x)$ having $0$ or negative slope.  
Recall that the spots in $\mathcal R$ all lie on or above edges of the 
Newton polygon of $F(x)$ with respect to $p$ which have a positive slope.  
Thus, the spots forming the endpoints of the edges of the 
Newton polygon of $F(x)$ having $0$ or negative slope must be among the spots
$\big(2n - i, \nu(c_{i})\big)$ where $2n-k+1 \le i \le 2n$.  Since 
$2n - (2n-k+1) = k-1 < \deg g$, these edges by themselves cannot consist of a complete
collection of translated edges of the Newton polygon of $g(x)$ with respect to $p$, and we
have a contradiction.  

Therefore, $F(x)$ cannot have a factor with degree $k$ or $k+1$.
\end{proof}

%%%%%%%%%%%%%%%%%%%%%%
\section{Proof of Lemma~\ref{3l2}}
Let $n$ be sufficiently large, and let $k$ be an odd integer in $[5,n]$.  Write
\[
(2n+1)(2n-1) \cdots (2n-k+2)=uv,
\]
where all the prime factors of $u$ are $\le k+2$ and all the prime factors
of $v$ are $> k+2$.  Then 
\[
v = \prod_{\substack{p^r \Vert ( (2n+1)(2n-1) \cdots (2n-k+2) ) \\ p \ge
 k+4 }} p^r.
 \]
To establish Lemma~\ref{3l2}, we want to show $v > 2n+1$.
We begin by establishing Lemma~\ref{3l2} when $k \ge 13$, 
and then handle other values of $k \ge 5$.  

\begin{lemma}\label{3l3}
For $n$ a sufficiently large integer and $13 \le k \le n$, 
\[
\prod_{\substack{p^r \Vert ( (2n+1)(2n-1) \cdots (2n-k+2) ) \\  p \ge  k+4 }} p^r>2n+ 1.
 \]
\end{lemma}

\begin{proof}
Let $T = \{ 2n+1, 2n-1, \dots , 2n-k+2 \}$.
In establishing Lemma \ref{3l3}, we take $n$ to be sufficiently large and
break up the argument into 3 cases depending on the size of $k$.

\vskip 8pt\noindent
\textit{Case 1: \ $n^{\frac{2}{3}} < k \le n$.}

We use the following lemma which follows from results on gaps between 
primes (for example, see \cite{gaps}).

\begin{lemma}\label{3lfix}
For $n$ sufficiently large and $n^{\frac{2}{3}} < k \le n$, there is a prime
in each of the intervals
\[
I_{1} = \bigg(2n - k +2, 2n + 1 - \frac{k-1}{2} \bigg] \quad \text{and} 
    \quad I_{2} = \bigg(2n+1 - \frac{k-1}{2}, 2n+1 \bigg].
    \]
\end{lemma}

By Lemma \ref{3lfix}, there exist primes $p_{1} \in I_{1}$ and $p_{2} \in 
I_{2}$.  Since 
\[ 
2n + 1 \ge p_{2} > p_{1} > 2n - k + 2,
\]
$p_{1}$ and $p_{2}$ are in $T$.
Since $k$ is odd, $p_{1} \ge 2n - k + 4$.
Thus (since $k \le n$), both $p_{1}$ and $p_{2}$ are $ \ge n + 4 \ge k + 4$.
From the definition of $v$, we deduce
\[
 v \ge p_{1} p_{2} > n^{2} > 2n+1.
 \]
This establishes Lemma \ref{3l3} for 
$n^{\frac{2}{3}} < k \le n$.

\vskip 8pt\noindent
\textit{Case 2: \ $k_{0} < k \le n^{\frac{2}{3}}$ 
for some fixed but sufficiently large $k_{0}$.}

For this case and indirectly later, we will make use of the following lemma which was established in \cite[Lemma~5]{af2}.

\begin{lemma}\label{T}
Let $m$, $\ell$ and $k$ denote positive integers with $k \ge 2$, and let
\[
T = \{ 2m+1, 2m+3, \dots, 2m+2\ell -1 \}.
\] 
For each odd prime $p \le k$ in turn,
remove from $T$ a number divisible by $p^e$ where $e = e(p)$ is as large as possible. 
Let $S$ denote the set of numbers that are left.  
Let $N_p$ be the exponent in the largest power of $p$ dividing $\prod_{t \in S} t$.
Then
\[
\prod_{p > k}p^{N_p} = 
\frac{\prod_{t \in S}t}{\prod_{2 < p \le k}p^{N_p}} \ge  
\frac{(2m+1)^{\ell - \pi(k) + 1}}{(\ell -1 )!}
\cdot 2^{\nu_2 (( \ell -1 )) ! }.
\]
\end{lemma}

For the moment, we only consider $1 \le k \le n^{\frac{2}{3}}$.  
To apply Lemma~\ref{T}, 
for each odd prime $p \le k+2$, consider a number in $T = \{ 2n+1, 2n-1, \dots , 2n-k+2 \}$ which is divisible
by $p^e$ where $e = e(p)$ is as large as possible.  Let $a_p$ denote such a number 
divisible by $p^e$.  Note that some of these numbers may be the same.
Dispose of all these numbers, and let $S$ denote the set of numbers in $T$ that are
left. 

Let $N_p$ be the exponent of the largest power of $p$ dividing
$\prod_{m \in S} m$.  Observe that
\[
v \ge \prod_{p > k+2}p^{N_p}.
\]
\noindent
By Lemma \ref{T} (with $2m + 1 = 2n - k + 2$, 
$\ell = (k+1)/2$, and $k$ replaced by $k + 2$), we obtain
\[
v \ge \frac{(2n-k+2)^{\frac{k+1}{2} - \pi(k+2) + 1}}{\big ( \frac{k-1}{2} \big )!}
\cdot 2^{\nu_2 \left( \left(\frac{k-1}{2} \right) ! \right)}.
\]
Let $r = (k+1)/2 - \pi(k+2) + 1$, and 
\[
\alpha_{k} =  \frac{\bigl (\frac{k-1}{2} \bigr ) !}
{2^{\nu_2  \left( \left(\frac{k-1}{2} \right) ! \right)}}.
\]
One can see that $T$ consists of $(k+1)/2$ numbers, and that we have removed 
at most $\pi(k+2) -1$ of them to obtain the set $S$.   Thus,
\begin{equation}\label{sizeSboundr}
|S| \ge r.
\end{equation}
Note that $v > 2n+1$ if 
\begin{equation}\label{3e1}
(2n-k+2)^r > \alpha_{k} (2n+1).
\end{equation}
Furthermore, \eqref{3e1} holds if 
\begin{equation}\label{3ef1}
r \log(2n - k +2) - (\log \alpha_{k} + \log(2n+1)) > 0.
\end{equation}
To reduce showing $v > 2n+1$ to establishing \eqref{3e1} or \eqref{3ef1}, we used $1 \le k \le n^{2/3}$.

We show now that \eqref{3ef1} holds for $k_{0} \le k \le n^{2/3}$
to finish the proof of Lemma~\ref{3l3} in this case.
Suppose then that $k_{0} \le k \le n^{2/3}$.
By the Prime Number Theorem, since $k_{0}$ is sufficiently large and
$k > k_{0}$, we have
\[
\pi(k+2) < \frac{1}{12} (k + 2).
\]
Hence, we deduce
\begin{equation}\label{3ef3}
\begin{aligned}
r &= \frac{k+1}{2} - \pi(k+2) + 1 > \frac{k+3}{2} - \frac{1}{12} (k+2) \\[5pt]
&> \frac{5}{12}(k+2) > \frac{5}{6} \bigg ( \frac{k-1}{2} \bigg ).
\end{aligned}
\end{equation}
Since $k \le n^{\frac{2}{3}}$, we obtain
\begin{equation}\label{3ef4}
2n - k + 2 > n.
\end{equation}
Also, the definition of $\alpha_{k}$ implies
\begin{equation}\label{3ef5}
\alpha_{k} \le \bigg (\frac{k-1}{2} \bigg ) ! \le k^{\frac{k-1}{2}} \le
       n^{\frac{2}{3} \left( \frac{k-1}{2} \right)}.
\end{equation}
By \eqref{3ef3} and \eqref{3ef4}, we deduce
\begin{equation}\label{3ef6}
r \log(2n - k +2) > \frac{5}{6} \bigg ( \frac{k-1}{2} \bigg ) \log n.
\end{equation}
From \eqref{3ef5}, we obtain
\begin{equation}\label{3ef7}
\log \alpha_{k} \le \frac{2}{3} \bigg ( \frac{k-1}{2} \bigg ) \log n.
\end{equation}
Combining \eqref{3ef6} and \eqref{3ef7} with $k \ge k_{0}$, we see that
\begin{equation}\label{3ef8}
\begin{aligned}
r \log(2n - k +2) - &(\log \alpha_{k} + \log(2n+1)) \\
&\quad > \frac{1}{6} \bigg( \frac{k-1}{2} \bigg) \log n - \log (2n+1) \\
&\quad \ge  \frac{1}{6} \bigg( \frac{k_{0}-1}{2} \bigg) \log n - \log (2n+1).
\end{aligned}
\end{equation}
Since $k_{0}$ is sufficiently large,
\[
\frac{1}{6} \bigg( \frac{k_{0}-1}{2} \bigg)
  \log n - \log (2n+1) > 0.
\]
Hence, \eqref{3ef1} holds, and $v > 2n + 1$.

\vskip 8pt\noindent
\textit{Case 3: \ $13 \le k \le k_{0}$.}

We show that \eqref{3e1} holds,
which will establish Lemma~\ref{3l3} in this case.
First, we obtain a lower bound for $r$.   
Since $k \ge 13$, the value of $\pi(k+2)$ is less than or equal to 
the number of even primes plus
the number of odd numbers less than or equal to $k+2$ 
minus the number of odd numbers less than or equal to $15$ that are not prime.  Thus,
\[
\pi(k+2) \le 1 + \frac{k+3}{2} - 3 \le \frac{k+3}{2}-2.
\]
Therefore,
\[
r=\frac{k+1}{2} - \pi(k+2) + 1  \ge \frac{k+1}{2}  - 
\frac{k+3}{2} + 3 =2.
\] 
By \eqref{sizeSboundr}, we deduce that there are at least two numbers in $S$ when $k \ge 13$.
Therefore, the left-hand side of \eqref{3e1}
is $\ge n^2$.  Observe that
\[
\alpha_{k} \le \bigg(\frac{k-1}{2} \bigg)! \le \bigg 
  (\frac{k_{0}-1}{2} \bigg ) !.
\]
Since $\alpha_{k}$ is bounded by a fixed constant, depending only on $k_{0}$, 
the right-hand side of \eqref{3e1} has order
$n$.  Therefore, for sufficiently large $n$, we see that \eqref{3e1} holds and, hence, $v > 2n+1$.
\end{proof}

\begin{lemma}\label{3l4}
For $k \in \{ 7, 9, 11 \}$, we have
\[
\prod_{\substack{p^r \Vert ( (2n+1)(2n-1) \cdots (2n-k+2) ) \\  p \ge k+4 }} p^r>2n+ 1
 \]
for all but finitely many positive integers $n$.
\end{lemma}

\begin{proof}
The first part of our proof (Case 1 below) will only involve finitely many exceptional $n$ which are $\le 13$.  
A direct check verifies the inequality in Lemma~\ref{3l4} for $14 \le n \le 42$,
and henceforth we only consider $n \ge 43$ throughout the proof of Lemma~\ref{3l4}.  For the second part of the proof (Case 2 below), the
exceptional $n$ can also be made explicit, but we do not do so and allow for these exceptional $n$ to be
somewhat larger.

Following the proof of Lemma~\ref{3l3}, Case 2, we set 
\[
T = \{ 2n+1, 2n-1, \dots , 2n-k+2 \},
\]
and apply Lemma~\ref{T}.  Thus, for each odd prime $p \le k+2$,
we remove from $T$ a number $a_{p}$ divisible by $p^e$, where $e = e(p)$ is as large as possible,
to obtain a subset $S$ of $T$ satisfying \eqref{sizeSboundr}, where 
\[
r = (k+1)/2 - \pi(k+2) + 1.
\]
A direct computation shows that $r = 1$ for $k \in \{ 7, 9, 11 \}$.  
Thus, $|S| \ge 1$. 
Fix a number $a = a(n,k)$ in $S$.  

\vskip 8pt\noindent
\textit{Case 1:  There exists a prime $p \ge k+4$ such that $p$ divides at least 1 of the numbers in $T - \{ a \} $.}

Recall $v$ is the product in Lemma~\ref{3l4}.  
If $k=7$, then $T= \{2n+1, 2n-1, 2n-3, 2n-5 \}$.  So $a \ge 2n-5$ and note
that $7 \nmid a$ and $5 \nmid a$ since the numbers $a_5$ and $a_7$
were removed from $T$ to form $S$.  Also, $3^2 \nmid a$.  Of
the numbers in $T - \{ a \}$, one of these is divisible by $p \ge 11$.  Thus,
\[
 v \ge \frac{p(2n-5)}{3} \ge \frac{11(2n-5)}{3}.
\]
A direct check shows that $11(2n-5)/3 > 2n+1$ if and only if $n > 29/8$.
Since $n \ge 43$, this last inequality holds, and $v > 2n+1$.

If $k = 9$, then $T= \{2n+1, 2n-1, 2n-3, 2n-5, 2n-7  \}$.  So $a \ge 2n-7$.
Note that $11 \nmid a$, $7 \nmid a$, and $5 \nmid a$.  Also $3^2 \nmid a$.
One of the numbers in $T - \{ a \} $ is divisible by $p \ge 13$.  Thus,
\[
v \ge \frac{p(2n-7)}{3} \ge \frac{13(2n-7)}{3}.
\]
As $13(2n-7)/3  > 2n+1$ if and only if $n > 47/10$, we deduce again that $v > 2n+1$.

If $k = 11$, then $T= \{2n+1, 2n-1, 2n-3, 2n-5, 2n-7, 2n-9  \}$.  So $a \ge 2n-9$.
Note that $13 \nmid a$, $11 \nmid a$, $7 \nmid a$, $5^2 \nmid a$, and $3^2 \nmid a$.
One of the numbers in $T - \{ a \} $ is divisible by $p \ge 17$.  Thus,
\[
v \ge \frac{p(2n-9)}{3 \times 5} \ge \frac{17(2n-9)}{15}.
\]
As $17(2n-9)/15  > 2n+1$ if and only if $n > 42$, 
we see again that $v > 2n+1$.

\vskip 8pt\noindent
\textit{Case 2:  There does not exist a prime $p \ge k+4$ such that $p$
divides at least one of the numbers in $T - \{ a \} $.}

To finish the proof of Lemma~\ref{3l4}, we prove that this case occurs for at most finitely many $n$. 
We make use of the following lemma which is a special case of a more
general theorem of Thue (see \cite{thue}).

\begin{lemma}\label{Thue}
Let $a$, $b$, and $c$ be fixed integers with $c \ne 0$.  Then there exist
only finitely many integer pairs $(x,y)$ for which 
$a x^3 + b y^3 = c.$
\end{lemma}
  
In this case, all the numbers in $T- \{ a \} $ are divisible only by primes $p \le k+2$. 
In particular, at least two of the numbers are powers of primes.  More precisely, when
$k=7$, one of $\{ 2n+1, 2n-1, 2n-3, 2n-5 \}$ is a power of 3 and another
a power of $p$ for some $p \in \{ 5, 7 \} $; when $k=9$, one
of $ \{ 2n+1, 2n-1, 2n-3, 2n-5, 2n-7 \} $ is a  power of $3$ and 
another a power of $p$ where $p \in \{ 5, 7, 11 \}$; and
when $k=11$, one of $ \{ 2n+1, 2n-1, 2n-3, 2n-5, 2n-7 , 2n-9 \} $ 
is a  power of $p_1$ and another a power of $p_2$ where $p_1 \ne p_2$ and
$p_1, p_2 \in \{ 3, 5, 7, 11, 13 \}$. 

Let $q$ be the greatest prime $ \le k+2 $, and 
let $P= \{ 3, 5, 7, \dots, q \}$ be the set of odd primes $\le q$.  
Let $2n-i$ and $2n-j$ where $i \ne j$ and
$i, j \in \{ -1, 1, 3, 5, 7, 9 \}$ denote two numbers
in $T - \{ a \} $ with each a power of a prime in $P$.  So $2n-i = p_1^u$ 
and $2n-j = p_2^v$ where $p_1, p_2 \in P$.  We consider the number of 
integer solutions $u$ and $v$ to 
\begin{equation}\label{3e2}
|p_1^u-p_2^v|=\ell \quad \text{where} \quad \ell \in \{ 2, 4, 6, 8, 10 \}.
\end{equation}
Note that by letting $u=3q_1 + r_1$ and $v=3q_2 + r_2$ where $q_1$ and 
$q_2$ are integers and $r_1, r_2 \in \{ 0, 1, 2 \} $, we can rewrite
the above equation as 
\begin{equation}\label{3e3}
|c_1x^3-c_2y^3|=\ell
\end{equation}
where $c_1$ and $c_2$ are integers and $x=p_1^{q_1}$ and $y=p_2^{q_2}$.
By Lemma~\ref{Thue}, there
are only finitely many integer solutions $x$ and $y$ to \eqref{3e3}, and thus 
finitely many integers $q_1$ and $q_2$. This implies that there are only
finitely many integers
$u$ and $v$ that satisfy \eqref{3e2}.  Therefore, there are only finitely many
integers $n$ where no number in $T - \{ a \} $  is divisible by a prime
$p \ge k+4$, completing what we set out to show for this case.
\end{proof}

Finally, we consider the case $k=5$.  In this case,
\[ 
v = \prod_{\substack{p^r \Vert ((2n+1)(2n-1)(2n-3)) \\  p \ge
 9 }} p^r.
 \]
We establish the following lemma, which will finish the proof of Lemma~\ref{3l2}.  
Furthermore, as noted after the statement of Lemma~\ref{3l2}, we will have completed the
proof of Theorem~\ref{themainthm}.

\begin{lemma}\label{3l5}
For $n$ a sufficiently large integer, we have
\[
\prod_{\substack{p^r \Vert ((2n+1)(2n-1)(2n-3)) \\  p \ge
 9 }} p^r>2n+ 1.
 \]
\end{lemma}

\begin{proof}
We use the following consequence of a theorem of Mahler \cite{mahler} which is demonstrated in \cite{fi}.

\begin{lemma}\label{mahler}
Let $a$ be a fixed non-zero integer, and let $N$ be a fixed positive integer.
Let $ \epsilon > 0$.  If $n$ is sufficiently large (depending on $a$, $N$,
and $\epsilon$), then the largest divisor of $n(n+a)$ which is relatively
prime to $N$ is $ \ge n^{1 - \epsilon}.$
\end{lemma}

Before proceeding, we note that Lemma~\ref{mahler} is ineffective, 
which in turn makes us unable to determine the finitely many exceptional pairs $(a_n,n)$ 
mentioned in Theorem~\ref{themainthm}.

Since we are only interested in the primes $p \ge 11$ that divide the product
$(2n+1)(2n-1)(2n-3)$, we take $N=3 \times 5 \times 7$ in Lemma \ref{mahler}.
Let
\[
A=\prod_{\substack{p^r \Vert (2n+1) \\  p \nmid
 N }} p^r, \ \ \ B=\prod_{\substack{p^r \Vert (2n-1) \\  p \nmid
 N }} p^r, \ \text{\ and \ } C=\prod_{\substack{p^r \Vert (2n-3) \\  
p \nmid N }} p^r.
\]
Note that
\[
\prod_{\substack{p^r \Vert ((2n+1)(2n-1)(2n-3)) \\  p \ge
 11 }} p^r = \prod_{\substack{p^r \Vert ((2n+1)(2n-1)(2n-3)) \\  p \nmid
 N }} p^r=ABC.
 \]
First consider the product $(2n+1)(2n-1)$.  Take $ \epsilon = 1/4$.  
Then by Lemma~\ref{mahler}, the largest divisor of $(2n+1)(2n-1)$ 
that is relatively prime to $N$ is $\ge (2n+1)^{3/4}$ for $n$ sufficiently large.  
Thus, $AB \ge (2n+1)^{3/4}$.  
We deduce that either $A \ge (2n+1)^{3/8}$ or $B \ge (2n+1)^{3/8}$. 

We suppose $A \ge (2n+1)^{3/8}$ (a similar argument can be done in the case that
$B \ge (2n+1)^{3/8}$).  Next, we consider the product $(2n-1)(2n-3)$.  
Again, take $ \epsilon = 1/4$.  
By Lemma \ref{mahler},
the largest divisor of $(2n-1)(2n-3)$ which is relatively prime to $N$ is
$\ge (2n-1)^{3/4}$ for $n$ sufficiently large.  
Thus, $BC \ge (2n-1)^{3/4}$, and we deduce
\begin{align*}
\prod_{\substack{p^r \Vert ((2n+1)(2n-1)(2n-3)) \\  p \ge 11 }} p^r &= ABC \ge (2n+1)^{3/8} (2n-1)^{3/4} \\
&> (2n-1)^{9/8} > (2n+1).
 \end{align*}
Therefore, for $n$ sufficiently large, we see that the inequality in Lemma~\ref{3l5} (and hence in Lemma~\ref{3l2}) holds.
\end{proof}

%%%%%%%%%%%%%%%%%%%%%%
\section{Establishing sharpness of the results}

We have shown that if $0 < |a_n| < M$ (where $M = \min\{ k',k''\}$), then $f(x)$
is irreducible.  We show that this upper bound on $|a_n|$ is sharp.  More precisely, 
we show that for $n>2$, when $a_n = M$ and $a_0 =1$, there exist integers 
$a_{n-1}, a_{n-2}, \dots, a_1$
such that $f(x)$ is reducible. 

Either $M=k' \le k''$ or $M=k''<k'$.  We show
the following results.

\begin{itemize}[leftmargin=15pt]
\item If $a_n = k'$ and $a_0 = 1$, then there are integers $a_{n-1}, a_{n-2}, \dots, a_1$ for which $f(x)$ has the irreducible 
quadratic factor $x^2 - 3$ (or similarly $x^{2}+3$).

\item If $a_n = k'' < k'$ and $a_0=1$, then there are integers $a_{n-1}, a_{n-2}, \dots, a_1$ for which $f(x)$ has the
irreducible quartic factor $x^4 - 5x^2 -15$.
\end{itemize}

We quickly address the situation where $n \le 2$ before restricting to $n > 2$.  
When $n=2$, the polynomial $f(x)$ is a quartic polynomial.  Also, $k' = 5$ and $k'' = 1$.  
From the comments after the statement of Lemma~\ref{3l2}, we deduce that
$f(x)$ cannot have a linear or quadratic factor
(and, thus, $f(x)$ is irreducible) whenever $n = 2$, $0 < |a_n| < 5$ and $|a_0|=1$.
Furthermore, when $n = 2$, $|a_n|=5$ and $|a_0|=1$, by the lemma below, there exists 
an integer $a_1$ such that $x^2 - 3$ (or $x^2+3$) is a factor of $f(x)$.
When $n=1$, we see that $k' = k'' = 1$ and $f(x) = a_{1} x^{2}/3 + a_{0}$, which is a quadratic polynomial,
and one can check that it is irreducible for $0 < |a_1| <3$ and $|a_0|=1$.  
For $|a_1|=3$ and $|a_0|=1$, with $a_{1}$ and $a_{0}$ of opposite signs, the quadratic $f(x)$ has 
the linear factors of $x+1$ and $x-1$.

For our goals above, we first show that, for $n \ge 2$, there exist integers $a_{n-1}, a_{n-2}, \ldots, a_1$ so that
we can make $x^2 - 3$ or $x^2 + 3$ (whichever we choose) a factor of $f(x)$
when $a_n = k'$ and $a_0 = 1$.

\begin{lemma}\label{3l9}
Let $n$ be an integer $\ge 2$, and
let $k'$ be the integer such that $2n+1 = k'3^u$ where $u$ is an
integer $\ge 0$ and $(k',3)=1$.  If 
$a_n=k'$ and $a_0=1$, then there exist integers 
$a_{n-1}, a_{n-2}, \dots, a_1$ such that
$x^2 - 3$ (or $x^2 + 3$) is a factor of $f(x)$.
\end{lemma}

\begin{proof}
Examples showing that the lemma holds for $n=2$ are given by
\begin{gather*}
u_6 f(x) = 5x^4 -20x^2 + 15 = 5 (x+1)(x-1)(x^2-3), \\
u_6 f(x) = 5x^4 + 20x^2 + 15 = 5 (x^{2}+1)(x^2+3).
\end{gather*}
We now consider $n \ge 3$.
Let $a_n = k'$, $a_0=1$, and $a_{n-2}=a_{n-3} = \cdots = a_2 = 0.$  Then
\[
u_{2n+2}f(x)=k'x^{2n}+a_{n-1}c_{n-1}x^{2n-2}+a_1c_1x^2+c_0
\]
where $c_{n-1}=2n+1=k' 3^u$, $c_0 = u_{2n+2}= 3 \times 5 \times
\cdots \times (2n-1) \times k'3^u$, and 
$c_1= u_{2n+2}/3=5 \times \cdots \times (2n-1) \times k'3^u$.  
Let $c= \nu_3(u_{2n+2})$.  Since $n \ge 3$, we have $c > u = \nu_3(c_{n-1})$.  Then
$c-1= \nu_3(u_{2n+2}/3)$.  Let $m$ be the integer for which
$c_0=u_{2n+2}=k'3^cm$ where $(3,m)=1$.  
Note that $k'$ and $m$ may not be coprime.
Let $t = x^2$ and consider the polynomial $F(t)$ where
\[
F(t) = k' t^n+a_{n-1}c_{n-1}t^{n-1}+a_1c_1t+c_0.
\]
We will obtain $t-3$ as a factor of $F(t)$, and thus $x^2 -3$ as a factor
of $f(x)$, by choosing $a_{n-1}$ and $a_1$ so that $F(3)=0$.

First, we show $c=\nu_3(u_{2n+2}) \le n-1$.
Since $n \ge 3$, we have $\nu_{3}(n!) \ge 1$.  Hence, we obtain
\begin{align*}
c &= \nu_3\big((2n+1)!\big)-\nu_3\big(2 \times 4 \times \cdots \times (2n)\big) \\
&= \nu_3\big((2n+1)!\big)-\nu_3\big(n!\big) \\
&= \sum_{j=1}^{\infty} \bigg\lfloor \dfrac{2n+1}{3^{j}} \bigg\rfloor - \nu_3\big(n!\big) \\
&< \sum_{j=1}^{\infty} \dfrac{2n+1}{3^{j}} - \nu_3\big(n!\big) \\
&= \dfrac{2n+1}{2} - 1 = n - (1/2).
\end{align*}
Since $c \in \mathbb Z$, we deduce $c \le n-1$.  

We are now ready to show that we can choose $a_{n-1}$ and $a_1$ so that $F(3)=0$.
Setting $m'= 3^{n-c} +m \in \mathbb Z$ and using our notation above, we see that
\begin{align*}
F(3) &= k' 3^n + a_{n-1}c_{n-1}3^{n-1} + a_1 c_1 3 + c_0 \\
&= k'3^n + a_{n-1} k' 3^{n+u-1} + a_1 k' 3^c m + k' 3^c m \\
&= k' 3^{c}(3^{n-c} + a_{n-1}3^{n+u-1-c} + a_1m+m) \\
&= k' 3^{c}(m' + a_{n-1}3^{n+u-1-c} + a_1m).
\end{align*}
Since $3^{n+u-1-c}$ and 
$m$ are relatively prime integers, there
exist integers $s$ and $t$ such that 
\[
3^{n+u-1-c}s+mt=1.
\]  
By taking $a_{n-1} = -sm'$ and $a_{1} = -tm'$, we deduce that $F(3) = 0$, as we wanted.

By a very similar analysis, one can show $x^2 + 3 $ can be a factor of $f(x)$ when
$a_n = k'$ and $a_0 = 1$, concluding the proof of Lemma~\ref{3l9}.
\end{proof}

We are left with considering the case that $a_n = k'' < k'$ and $a_0=1$.  
We restrict to $n \ge 3$, and note that $k'' < k'$ implies then that $n \ge 12$. 
The definitions of $k'$ and $k''$ further give that if $k '' < k'$, 
then $5 \mid (2n+1)$.  Also, we see that
$3 \nmid (2n+1)$ since otherwise $k' \le (2n+1)/3 < 2n-1 \le k''$.
Hence, we can write $2n-1 = 3^{k} m$ and $2n+1 = 5^{\ell} m'$ where $k$ is a 
nonnegative integer and $n$, $\ell$, $m$, and $m'$ are positive integers with $n \ge 12$
and $\gcd(mm',15) = 1$.  Observe that $k'' = mm'$.  
With this notation, set $a_{n} = mm'$.  We show
that there exist integers $a_{n-1},a_{n-2}, \dots,a_{1}$ such that
the polynomial
\[
f(t)=a_{n} \dfrac{x^{2n}}{u_{2n+2}} + a_{n-1} \dfrac{x^{2n-2}}{u_{2n}} 
+ \cdots + a_{1} \dfrac{x^2}{3} + 1
\]
has the quartic factor $x^4 - 5 x^2 -15$.
Let $t = x^{2}$, and let
\[ 
F(t) = a_{n} \dfrac{t^{n}}{u_{2n+2}} + a_{n-1} \dfrac{t^{n-1}}{u_{2n}} 
+ \cdots + a_{1} \dfrac{t}{3} + 1. 
\]
Note that $F(x^2) = f(x)$.  Thus, it suffices to show that $F(t)$
is divisible by the quadratic $q(t) = t^{2} - 5 t - 15$. 
To do this, we multiply $F(t)$ by $u_{2n+2}$ and divide through by $a_{n} = m m'$ to obtain the polynomial
\begin{align*} 
t^{n} + 5^{\ell} \dfrac{a_{n-1}}{m} t^{n-1} &+ 5^{\ell} 3^{k} a_{n-2} t^{n-2} 
+ \cdots \\ &\quad + 5^{\ell-1} 3^{k-1} u_{2n-2} a_{2} t^{2} + 5^{\ell} 
3^{k-1} u_{2n-2} a_{1} t 
+ 5^{\ell} 3^{k} u_{2n-2}. 
\end{align*}
Let $r$, $s$, $y$ and $w$ be variables representing integers, and take
$a_{n-1} = m r$, $a_{n-2} = s$, $a_{n-3} = a_{n-4} = \cdots = a_{3} = 0$,
$a_{2} = -y$, and $a_{1} = w+y$.  The polynomial above becomes
\begin{align*} 
g(t) = t^{n} + 5^{\ell} r t^{n-1} &+ 5^{\ell} 3^{k} s t^{n-2} - 5^{\ell-1} 
3^{k-1} u_{2n-2} y t^{2}  \\
&\quad + 5^{\ell} 3^{k-1} u_{2n-2} (w+y) t + 5^{\ell} 3^{k} u_{2n-2}. 
\end{align*}
It suffices now to show that there exist integers $r$, $s$, $y$, and $w$ 
such that 
$g(t)$ is divisible by $q(t)$.  

For $j \ge 0$, define integers $b_{j}$ and $c_{j}$ by
\[ 
t^{j} \equiv b_{j}t+c_{j} \pmod{q(t)}. 
\]
Since
\begin{equation}\label{eqone} 
t^{j+1} \equiv 5 t^{j}+15 t^{j-1} \pmod{q(t)}, \quad \text{ for } j \ge 1,
\end{equation}
we deduce
\begin{equation}\label{eqtwo}
c_{j+1} = 5 c_{j}+15 c_{j-1} 
\qquad \text{and} \qquad b_{j+1} = 5 b_{j}+15 b_{j-1}, \quad \text{ for } j \ge 1.
\end{equation}
Letting
\[ 
A = \begin{pmatrix}
0 & 1 \\
15 & 5 
\end{pmatrix},
\]
we obtain from \eqref{eqtwo} and an induction argument that
\[ 
A^{j} = \left( \begin{array}{cc}
c_{j} & b_{j} \\
c_{j+1} & b_{j+1} 
\end{array} \right),  \quad \text{ for } j \ge 0.
\]
Next, we obtain some results for the values of $\nu_{3}(c_{j})$, 
$\nu_{3}(b_{j})$, 
$\nu_{5}(c_{j})$, and
$\nu_{5}(b_{j})$.  
An induction argument gives that
\[ 
A^{2j} \equiv \begin{pmatrix}
6 & 5 \\
3 & 4 
\end{pmatrix} {\hskip -6pt}\pmod{9}
\quad \text{ and } \quad
A^{2j+1} \equiv \begin{pmatrix}
3 & 4 \\
6 & 5 
\end{pmatrix} {\hskip -6pt}\pmod{9},  \quad \text{ for } j \ge 1.
\]
Hence, we see that
\begin{equation}\label{eqthree}
\nu_{3}(c_{j}) = 1 \quad \text{and} \quad \nu_{3}(b_{j}) = 0,  \quad \text{ for } j \ge 2.
\end{equation}

Next, we claim that 
\begin{equation}\label{eqfour}
\nu_{5}(c_{j}) \ge \dfrac{j}{2} \quad \text{and} \quad 
\nu_{5}(b_{j}) \ge \dfrac{j-1}{2},  \quad \text{ for } j \ge 2.
\end{equation}
For $j=2$ and $j=3$, one checks directly that \eqref{eqfour} holds.
From \eqref{eqtwo}, we deduce
\begin{align*}
\nu_{5}(c_{j+1}) &\ge \min\{\nu_{5}(c_{j}), \nu_{5}(c_{j-1})\} + 1 \\
\intertext{and} 
\nu_{5}(b_{j+1}) &\ge \min\{\nu_{5}(b_{j}), \nu_{5}(b_{j-1})\} + 1. 
\end{align*}
By induction, we obtain that \eqref{eqfour} holds.

Using that $\det (A^{j}) = \det(A)^{j}$, we obtain
\begin{equation}\label{eqfive}
c_{j}b_{j+1} - c_{j+1}b_{j} = \pm 15^{j}. 
\end{equation}
Given \eqref{eqfour}, we deduce that, for $j \ge 2$,
at least one of $\nu_{5}(c_{j}) = j/2$ and $\nu_{5}(c_{j+1}) = (j+1)/2$
holds.  Only one of $j/2$ and $(j+1)/2$ can be an integer.  It follows that
\begin{equation}\label{eqsix}
\nu_{5}(c_{j}) = \dfrac{j}{2} \quad \text{ for } j \ge 2 \text{ even}.
\end{equation}
Note that parity considerations also imply from \eqref{eqfour} that
$\nu_{5}(c_{j}) \ge (j+1)/2$ if $j$ is odd and that
$\nu_{5}(b_{j}) \ge j/2$ if $j$ is even. 

Recall that $t^{2} \equiv 5 t+15 \pmod{q(t)}$.  We obtain from the
definitions $g(t)$, $b_{j}$ and $c_{j}$ that
\begin{align*} 
g(t) \equiv \big( b_{n} + 5^{\ell} r b_{n-1} &+ 5^{\ell} 3^{k} s b_{n-2} 
+ 5^{\ell} 3^{k-1} w u_{2n-2} \big) t \\
&\quad + 
c_{n} + 5^{\ell}r c_{n-1} + 5^{\ell} 3^{k} s c_{n-2} + 5^{\ell} 3^{k} 
u_{2n-2} (1-y)
\end{align*}
modulo $q(t)$. 
We will show that for some integers $r$, $s$, $y$, and $w$, we have
\[
b_{n} + 5^{\ell} r b_{n-1} + 5^{\ell} 3^{k} s b_{n-2} + 5^{\ell} 3^{k-1} w 
u_{2n-2}=0.
\]
and
\[
c_{n} + 5^{\ell}r c_{n-1} + 5^{\ell} 3^{k} s c_{n-2} + 5^{\ell} 3^{k} 
u_{2n-2} (1-y)=0
\]
It will then follow that $g(t) \equiv 0 \pmod{q(t)}$.

We first show that there are integers $r$, $s$, and $y$ such that
\[
5^{\ell}r c_{n-1} + 5^{\ell} 3^{k} s c_{n-2}=-(c_n+5^{\ell} 3^{k} u_{2n-2} 
(1-y)).
\]
Observe that $n \ge 12$ and \eqref{eqtwo} imply $c_{n-1} > 0$ and $c_{n-2} > 0$.  
Since in general, the Diophantine equation $ax+by = c$ for fixed positive integers $a$ and $b$
and for an ineteger $c$ has solutions in integers $x$ and $y$ if and only if $\gcd(a,b) \mid c$,
the above equation in $r$, $s$, and $y$ will have integer solutions in $r$ and $s$ if we can choose $y$ so that
\begin{equation}\label{soughtgcd}
\gcd( 5^{\ell}c_{n-1}, 5^{\ell}3^{k}c_{n-2}) \mid (c_n+5^{\ell} 
3^{k} u_{2n-2} (1-y)).
\end{equation}
Since $2n+1 = 5^{\ell} m'$ and
\[ 
\nu_{5}((2n+1)!) < \dfrac{2n+1}{5} + \dfrac{2n+1}{5^{2}} + 
\dfrac{2n+1}{5^{3}} + \cdots
= \dfrac{2n+1}{4}, 
\] 
we obtain 
\begin{equation}\label{nufiveest}
\nu_{5}(5^{\ell} 3^{k} u_{2n-2}) \le n/2.
\end{equation}
Also, \eqref{eqfour}
implies $\nu_{5}(c_{n}) \ge n/2$.  It follows that, 
by choosing $1-y$ to satisfy a congruence modulo a sufficiently large power of $5$, 
there is an integer $y$ such that 
\[ 
\nu_{5}\big( c_{n} + 5^{\ell} 3^{k} u_{2n-2} (1-y) \big) \ge  
\ell + \min\{ \nu_{5}(c_{n-1}), \nu_{5}(c_{n-2}) \}. 
\]
We may also find such a $y$ with $1-y$ divisible by $3$.  
Fix such a $y$.  From \eqref{eqthree}, we obtain
\begin{align*}
\nu_{3}\big( c_{n} + 5^{\ell} 3^{k} u_{2n-2} (1-y) \big) &\ge 1 = \nu_{3}(5^{\ell}c_{n-1}) \\
&\ge \min\{ \nu_{3}(5^{\ell}c_{n-1}), \nu_{3}(5^{\ell}3^{k}c_{n-2}) \}. 
\end{align*}
Note that \eqref{eqfive} implies that $3$ and $5$ are the only prime 
factors possibly in common with $c_{n-1}$ and $c_{n-2}$.  
We deduce that \eqref{soughtgcd} holds, and therefore
there exist integers $r_{0}$ and $s_{0}$ such that 
\begin{equation}\label{eqseven} 
c_{n} + 5^{\ell}r_{0} c_{n-1} + 5^{\ell} 3^{k} s_{0} c_{n-2} 
+ 5^{\ell} 3^{k} u_{2n-2} (1-y) = 0.
\end{equation}

We fix $r_{0}$ and $s_{0}$ as above and observe that, 
for every integer $t$, we have
\[ 
c_{n} + 5^{\ell} c_{n-1} \big( r_{0} + 3^{k}c_{n-2} t \big) 
+ 5^{\ell} 3^{k} c_{n-2} \big( s_{0} - c_{n-1} t \big) 
+ 5^{\ell} 3^{k} u_{2n-2} (1-y) = 0. 
\]
We set 
\[ 
r = r_{0} + 3^{k}c_{n-2} t 
\qquad \text{and} \qquad
s = s_{0} - c_{n-1} t 
\]
and seek $t$ and $w$ so that
\[ 
b_{n} + 5^{\ell} r b_{n-1} + 5^{\ell} 3^{k} s b_{n-2} 
+ 5^{\ell} 3^{k-1} w u_{2n-2} = 0. 
\]
In other words, we want
\begin{align*}
5^{\ell} 3^{k-1} w u_{2n-2} &+ 5^{\ell} 3^{k}
\big( c_{n-2}b_{n-1} - c_{n-1}b_{n-2} \big) t \\ 
&\qquad + b_{n} + 5^{\ell}r_{0} b_{n-1} + 5^{\ell} 3^{k} s_{0} b_{n-2} = 0. 
\end{align*}
By \eqref{eqtwo}, we can rewrite this equation as
\begin{equation}\label{eqeight}
\begin{aligned}
5^{\ell} 3^{k-1} w u_{2n-2} &+ 5^{\ell} 3^{k}
\big( c_{n-2}b_{n-1} - c_{n-1}b_{n-2} \big) t  \\  
&\quad + (5^{\ell}r_{0}+5) b_{n-1} + (5^{\ell} 3^{k} s_{0} +15) b_{n-2} = 0. 
\end{aligned}
\end{equation}
From \eqref{eqtwo} and \eqref{eqseven}, we obtain
\begin{equation}\label{eqnine}
(5^{\ell}r_{0}+5) c_{n-1} + (5^{\ell} 3^{k} s_{0} +15) c_{n-2} 
+ 5^{\ell} 3^{k} u_{2n-2} (1-y) = 0. 
\end{equation}
From $n \ge 12$ and \eqref{eqtwo}, we see that each of $b_{n-2}$, $b_{n-1}$, 
$c_{n-2}$ and $c_{n-1}$ is nonzero.  
Multiplying both sides of \eqref{eqeight} by $c_{n-1}$ and both sides
of \eqref{eqnine} by $-b_{n-1}$ and then adding, we obtain
\begin{equation}\label{eqten}
\begin{aligned}
c_{n-1}5^{\ell} 3^{k-1} u_{2n-2} w &+ c_{n-1} 5^{\ell} 3^{k}
\big( c_{n-2}b_{n-1} - c_{n-1}b_{n-2} \big) t \\
&- (5^{\ell} 3^{k} s_{0} +15) \big( c_{n-2}b_{n-1} - c_{n-1}b_{n-2} 
\big)  \\
&- 5^{\ell} 3^{k} u_{2n-2} (1-y) b_{n-1} = 0. 
\end{aligned}
\end{equation}
Multiplying both sides of \eqref{eqeight} by $c_{n-2}$ and both sides
of \eqref{eqnine} by $-b_{n-2}$ and then adding, we obtain
\begin{equation}\label{eqeleven}
\begin{aligned}
c_{n-2}5^{\ell} 3^{k-1} u_{2n-2} w &+ c_{n-2} 5^{\ell} 3^{k}
\big( c_{n-2}b_{n-1} - c_{n-1}b_{n-2} \big) t \\
&+ (5^{\ell}r_{0}+5) \big( c_{n-2}b_{n-1} - c_{n-1}b_{n-2} \big) \\
&- 5^{\ell} 3^{k} u_{2n-2} (1-y) b_{n-2} = 0. 
\end{aligned}
\end{equation}
Observe that \eqref{eqnine} implies that if either \eqref{eqten} or 
\eqref{eqeleven} holds, 
then so does \eqref{eqeight}.
Recall \eqref{eqsix} and the comment after it.  
We work with \eqref{eqten} if $n$ is odd and make use of
$\nu_{5}(c_{n-1}) = (n-1)/2$, $\nu_{5}(c_{n-2}) \ge (n-1)/2$, and
$\nu_{5}(b_{n-1}) \ge (n-1)/2$.  For $n$ even, we work with
\eqref{eqeleven} and make use of
$\nu_{5}(c_{n-1}) \ge n/2$, $\nu_{5}(c_{n-2}) = (n-2)/2$, and
$\nu_{5}(b_{n-2}) \ge (n-2)/2$.    
We give the details of the argument in the case that $n$ is odd, and
simply note that a similar argument works in the case $n$ is even.

Fix $n \ge 12$ odd. 
Let 
\begin{gather*} 
c = c_{n-1}5^{\ell} 3^{k-1} u_{2n-2}, \quad 
c' = c_{n-1} 5^{\ell} 3^{k} \big( c_{n-2}b_{n-1} - c_{n-1}b_{n-2} \big), \\
c'' = (5^{\ell} 3^{k} s_{0} +15) \big( c_{n-2}b_{n-1} - c_{n-1}b_{n-2} \big), \\
\intertext{and}
c''' = 5^{\ell} 3^{k} u_{2n-2} (1-y) b_{n-1}.
\end{gather*}
From \eqref{nufiveest}, we deduce 
\[ 
\nu_{5}(c) \le \dfrac{n}{2} + \dfrac{n-1}{2} = n-\dfrac{1}{2}
\quad \implies \quad \nu_{5}(c) \le n-1. 
\]
Observe that
\[ 
\nu_{3}((2n-1)!) < \dfrac{2n-1}{3} + \dfrac{2n-1}{3^{2}} + 
\dfrac{2n-1}{3^{3}} + \cdots 
= \dfrac{2n-1}{2} = n - \dfrac{1}{2}. 
\]
Since 
$2n-1 = 3^{k} m$, we see that
\begin{align*}
\nu_{3}(3^{k-1} u_{2n-2}) &= \nu_{3}(u_{2n})-1 \\
&= \nu_{3}((2n-1)!) - \nu_{3}(2 \times 4 \times \cdots \times (2n-2)) -1.
\end{align*}
Since $n \ge 12$, we have $\nu_{3}(2 \times 4 \times \cdots \times (2n-2)) \ge 
\nu_{3}(2 \times 4 \times \cdots \times 22) = 4$.  Thus,    
$\nu_{3}(3^{k-1} u_{2n-2}) \le \nu_{3}((2n-1)!) - 5$. 
From \eqref{eqthree}, we see that 
\[
\nu_{3}(c) \le \nu_{3}(c_{n-1}) + \nu_{3}(3^{k-1} u_{2n-2})
\le \nu_{3}((2n-1)!)-4 < n -4.
\]
Since $\nu_{3}(c)$ is an integer, we obtain
\[ \nu_{3}(c) \le n-5. \]
Since $\ell \ge 1$, we have $5$ divides $5^{\ell} 3^{k} s_{0} +15$.  We obtain from \eqref{eqfive}
 that
\[ 
\nu_{3}(c'') \ge n-2 \qquad \text{and} \qquad \nu_{5}(c'') \ge n-1.
\]
Observe that \eqref{eqthree} implies $\nu_{3}(c''') \ge \nu_{3}(c)$.  Since 
$n$ is odd, we obtain from \eqref{eqsix} and the comment after it that
\[ 
\nu_{5}(b_{n-1}) \ge \dfrac{n-1}{2} = \nu_{5}(c_{n-1}). 
\]
Hence, $\nu_{5}(c''') \ge \nu_{5}(c)$.  Combining the above, we deduce
\[ \nu_{3}(c'' + c''') \ge \nu_{3}(c) \qquad \text{and} \qquad
 \nu_{5}(c'' + c''') \ge \nu_{5}(c). \]

We claim that $\gcd(c,c')$ divides $c'' + c'''$.  Let $p$ be a prime and 
$u$ a positive integer
for which $p^{u}\Vert \gcd(c,c')$.  The above analysis shows that if $p = 3$ or 
$p = 5$, then
$p^{u} \mid (c'' + c''')$.  Now, consider the case that $p$ is not $3$ or $5$.  
From \eqref{eqfive}
and the definition of $c'$, we obtain that $p^{u} \mid c_{n-1}$.  
From \eqref{eqnine}, we see that $p^{u}$ must also divide
\[ 
\big( (5^{\ell}3^{k}s_{0}+15) c_{n-2} + 5^{\ell}3^{k}u_{2n-2} (1-y) \big) 
b_{n-1}
- (5^{\ell}3^{k}s_{0}+15) b_{n-2} c_{n-1},
\] 
which is the same as $c'' + c'''$.  
Hence, $\gcd(c,c')$ divides $c'' + c'''$.

It now follows that there exist integers $w$ and $t$ for which 
$c w + c' t = c'' + c'''$.  This establishes the existence of integers $w$ 
and $t$
as in \eqref{eqeight} and, hence, the existence of integers $r$, $s$, 
$y$, and $w$ 
for which $g(x)$ is divisible by $x^{2}-5x-15$.

\vskip 5pt
\centerline{\hrulefill\phantom{have a nice day}}
\vskip 5pt\noindent
\parbox[t]{13.2cm}{
Department of Mathematics and Computer Science, Georgia College and State University, Milledgeville, GA  31061, U.S.A.;
Email: martha.allen{@}gcsu.edu.
}

\vskip 10pt\noindent
\parbox[t]{13.2cm}{
Mathematics Department, University of South Carolina, Columbia, SC  29208, U.S.A.;
Email: filaseta{@}math.sc.edu; 
http://www.math.sc.edu/{\,\~{}}filaseta/.
}


\begin{thebibliography}{10}

\bibitem{ma} M.~Allen, 
Generalizations of the Irreducibility Theorems of I.~Schur, 
Dissertation from the University of South Carolina, 2001.  

\bibitem{af} M.~Allen and M.~Filaseta, \textit{A generalization of a second irreducibility
theorem of I.~Schur}, Acta Arith.~109 (2003), 65--79.

\bibitem{af2}
M.~Allen and M.~Filaseta,
\textit{A generalization of a third irreducibility theorem of {I}.~{S}chur},
Acta Arith.~114 (2004), 183--197.

%% \bibitem{led} L. E. Dickson, History of the Theory of Numbers, Vol. II, Chelsea, NY, 1971, 744.

\bibitem{d} G.~Dumas, \textit{Sur quelques cas d'irr\'eductibilit\'e des
polyn\^omes \`a coefficients rationnels}, Journal de Math.~Pures et Appl., 2 (1906), 191--258.

%% \bibitem{eees} E.~F.~Ecklund, Jr., R.~B.~Eggleton, P.~Erd\H os, and J.~L.~Selfridge, \textit{On the prime factorization of binomial coefficients}, J.~Austral.~Math.~Soc., (Series A) 26 (1978), 257--269.

%% \bibitem{pe} P.~Erd\H os, \textit{On consecutive integers}, Nieuw Arch.~Wisk.~(3), 3 (1955), 124--128.

\bibitem{fi} M.~Filaseta, \textit{A generalization of an irreducibility
theorem of I.~Schur}, in: Analytic Number Theory, Proc.~Conf.~in Honor of Heini Halberstam, vol.~1, 
B.~C.~Berndt, H.~G.~Diamond, and A.~J.~Hildebrand (eds.), Birkh\"auser, Boston, 
1996, 371--395.

%% \bibitem{ft} M.~Filaseta and O.~Trifonov, \textit{The irreducibility of the Bessel Polynomials},  J.~Reine Angew.~Math.~550 (2002), 125--140.

\bibitem{gaps} M. N. Huxley, 
\textit{On the difference between consecutive primes}, Invent. Math.~15 (1972), 155--164.

%% \bibitem{dl} D.~H.~Lehmer, \textit{On a problem of St\"ormer}, Illinois J.~Math.~8 (1964), 57--79.

\bibitem{jakhar}
A.~Jakhar, \textit{On Schur's irreducibility results and generalised $\varphi$-Hermite polynomials}, 
arXiv:2306.01767.

\bibitem{mahler} K.~Mahler,
Lectures on Diophantine Approximations, Cushing-Malloy, Ann Arbor, 1961.  

\bibitem{thue} L. J. Mordell, Diophantine Equations,  Academic Press,
London, 1969.

%% \bibitem{ros}J.~B.~Rosser and L.~Schoenfeld, \textit{Sharper bounds for Chebyshev functions $\theta(x)$ and $\psi(x)$}, Math.~Comp., 29 (1975), 243--269.

\bibitem{sc} I.~Schur,
\textit{Einige S\"atze \"uber Primzahlen mit Anwendungen auf Irreduzibilit\"atsfragen, I},
Sitzungsber.~Preuss.~Akad.~Wiss.~Berlin Phys.-Math.~Kl., 14 (1929), 125--136.

\bibitem{sch} I.~Schur, \textit{Einige S\"atze \"uber Primzahlen mit Anwendungen 
auf Irreduzibilit\"atsfragen, II},
Sitzungsber.~Preuss.~Akad.~Wiss.~Berlin Phys.-Math.~Kl., 
14 (1929), 370--391.

\end{thebibliography}
\end{document}